\documentclass[reqno,11pt]{amsart}
\usepackage[margin=1in]{geometry}

\usepackage{amsthm, amsmath, amssymb, amsbsy, asymptote, bm}
\usepackage{makecell, tikz}
\usepackage{mathdots}

\usepackage{thmtools}
\usepackage{thm-restate}

\usepackage{microtype} %turn on at the end -- slow

\usepackage[utf8]{inputenc}
\usepackage[T1]{fontenc}

\usepackage[textsize=small,backgroundcolor=orange!20]{todonotes}

\usepackage[numbers]{natbib}
\usepackage{hyperref}
\usepackage{url}
\usepackage{doi}

\usepackage{xcolor,graphics}
\begin{asydef}
// Global Asymptote definitions can be put here.
usepackage("mathdots");
\end{asydef}

\newtheorem{theorem}{Theorem}[section]
\newtheorem{proposition}[theorem]{Proposition}
\newtheorem{lemma}[theorem]{Lemma}

\newtheorem{conjecture}[theorem]{Conjecture}
\newtheorem*{question*}{Question}

\newtheorem{case}{Case}
\newtheorem{case2}{Case}

\theoremstyle{definition}
\newtheorem{definition}[theorem]{Definition}

\newtheorem{question}[theorem]{Question}

\theoremstyle{remark}

\tikzstyle{P} = [draw, circle, black, fill, inner sep = 0pt, minimum width = 3pt]
\tikzstyle{every loop} = []

\title{Improving the $\frac{1}{3}-\frac{2}{3}$ Conjecture for Width Two Posets}

\author[Sah]{Ashwin Sah}
\address{Massachusetts Institute of Technology, Cambridge, MA 02139, USA}
\email{asah@mit.edu}

\begin{document}

\begin{abstract}
Extending results of Linial (1984) and Aigner (1985), we prove a uniform lower bound on the balance constant of a poset $P$ of width $2$. This constant is defined as $\delta(P) = \max_{(x, y)\in P^2}\min\{\mathbb{P}(x\prec y), \mathbb{P}(y\prec x)\}$, where $\mathbb{P}(x\prec y)$ is the probability $x$ is less than $y$ in a uniformly random linear extension of $P$. In particular, we show that if $P$ is a width $2$ poset that cannot be formed from the singleton poset and the three element poset with one relation using the operation of direct sum, then
\[\delta(P)\ge\frac{-3 + 5\sqrt{17}}{52}\approx 0.33876\ldots.\]
This partially answers a question of Brightwell (1999); a full resolution would require a proof of the $\frac{1}{3}-\frac{2}{3}$ Conjecture that if $P$ is not totally ordered then $\delta(P)\ge\frac{1}{3}$.

Furthermore, we construct a sequence of posets $T_n$ of width $2$ with $\delta(T_n)\rightarrow\beta\approx 0.348843\ldots$, giving an improvement over a construction of Chen (2017) and over the finite posets found by Peczarski (2017). Numerical work on small posets by Peczarski suggests the constant $\beta$ may be optimal.
\end{abstract}

\maketitle

\section{Introduction}\label{sec:introduction}
\begin{definition}\label{def:pair-probability}
Given a fixed, underlying poset $(P, \le)$, $\mathbb{P}(x\prec y)$ is the probability that $x$ precedes $y$ in a uniformly random linear extension of $P$. We define $\mathbb{P}(x\prec x) = 0$.
\end{definition}

A conjecture dating back to 1968 states that in any finite partial order not a chain, there is a pair $(x, y)$ such that $\mathbb{P}(x\prec y)\in\left[\frac{1}{3}, \frac{2}{3}\right]$. Kislitsyn \cite{kislitsyn1968finite}, Fredman \cite{fredman1976good}, and Linial \cite{linial1984information} independently made this so-called $\frac{1}{3}-\frac{2}{3}$ Conjecture. Each had in mind an application to sorting theory. In particular, this conjecture implies that the number of comparisons needed to fully sort elements that are already known to be in the partial order $P$ is at most $(1+o(1))\log_{\frac{3}{2}} e(P)$, within a constant factor of the trivial information-theoretic lower bound $\log_2 e(P)$. Here $e(P)$ is the number of linear extensions of $P$.

\begin{definition}\label{def:balance-constant}
The \emph{balance constant} of poset $P$ is
\[\delta(P) = \max_{(x, y)\in P^2}\min\{\mathbb{P}(x\prec y), \mathbb{P}(y\prec x)\}.\]
\end{definition}
We can thus restate the $\frac{1}{3}-\frac{2}{3}$ Conjecture as follows.
\begin{conjecture}
If $P$ is a finite poset that is not totally ordered, then $\delta(P)\ge\frac{1}{3}$.
\end{conjecture}
Brightwell \cite{brightwell1999balanced} deemed it ``one of the major open problems in the combinatorial theory of partial orders.''

This conjecture is known to be true for certain classes of posets: width $2$ by Linial \cite{linial1984information}, height $2$ by Trotter, Gehrlein, and Fishburn \cite{trotter1992balance}, $6$-thin by Peczarski \cite{peczarski2008gold}, semiorders by Brightwell \cite{brightwell1989semiorders}, $N$-free posets by Zaguia \cite{zaguia20121}, and posets whose Hasse diagram is a forest by Zaguia \cite{zaguia20191}.

We can ask how the balance constant interacts with the fundamental operations of disjoint union and direct sum. (Direct sum will be important in characterizing equality cases.) It is clear that if $\oplus$ denotes direct sum of posets and $\sqcup$ disjoint union of posets then
\begin{align*}
\delta(P\oplus Q) &= \max\{\delta(P), \delta(Q)\},\\
\delta(P\sqcup Q)&\ge\max\{\delta(P), \delta(Q)\}.
\end{align*}
We more formally define $P\oplus Q$ on the union set such that $x\le y$ precisely when $x\le_P y$ or $x\le_Q y$ or both $x\in P$ and $y\in Q$. Also, $P\sqcup Q$ is defined on the union such that $x\le y$ precisely when $x\le_P y$ or $x\le_Q y$. Thus the $\frac{1}{3}-\frac{2}{3}$ Conjecture is true for direct sums and disjoint unions of posets which satisfy it. It follows easily that the conjecture holds, for instance, for series-parallel posets, i.e.\ $N$-free posets. Brightwell, Felsner, and Trotter \cite{brightwell1995balancing} showed that if $P$ is not totally ordered, then $\delta(P)\ge\frac{5 - \sqrt{5}}{10}$, improving on methods of Kahn and Saks \cite{kahn1984balancing}. See the survey of Brightwell \cite{brightwell1999balanced} for more information on general progress.

Aigner \cite{aigner1985note} showed that the only width $2$ posets that achieve equality ($\delta(P) = \frac{1}{3}$) are those formed from $\mathbf{1}$ and $\mathcal{E}$ using the operation of direct sum: $\mathbf{1}$ is the poset with one element and $\mathcal{E}$ is the poset with three elements and exactly one relation $x\le y$. Alternatively, $\mathcal{E} = (\mathbf{1}\oplus\mathbf{1})\sqcup\mathbf{1}$.

Brightwell \cite{brightwell1999balanced} posed the question of understanding in general the structure of the set $\mathcal{B} = \{\delta(P): P\text{ a finite partial order}\}$, asking whether there is a gap after $\frac{1}{3}$. Of course, a result of this form would be much stronger than the $\frac{1}{3}-\frac{2}{3}$ Conjecture.

We answer this question in the affirmative in the width $2$ setting, thus extending the results of Aigner \cite{aigner1985note} and Linial \cite{linial1984information}. In particular, we prove
\begin{restatable}{theorem}{thmmain}\label{thm:main}
If $P$ is a finite, width $2$ poset that cannot be formed from $\mathbf{1}$ and $\mathcal{E}$ using the operation of direct sum, then
\[\delta(P)\ge\lambda = \frac{-3 + 5\sqrt{17}}{52} =  0.33876\ldots.\]
\end{restatable}
The proof relies on a path-counting interpretation of linear extensions of width $2$ posets, and as noted in Section~\ref{sec:concluding}, computer results seem to indicate that we can improve the constant, potentially to $0.348842$ or so.

On the other side of the issue, Chen \cite{chen2018family} exhibited a sequence of width $2$ posets whose balance constants approach $\frac{93 - \sqrt{6697}}{32}\approx 0.3488999\ldots$. Using our path-counting interpretation, we can easily compute balance constants of similar families, and in particular show
\begin{theorem}\label{thm:construction}
There is a sequence $T_n$ of width $2$ posets with
\[\delta(T_n)\rightarrow\beta = \frac{5864893 + 27\sqrt{57}}{16812976}\approx 0.348843\ldots.\]
\end{theorem}
It is worth noting the similarity of our results to some known results about poset entropy. Let $\overline{G}(P)$ be the incomparability graph of $P$, with vertex set $P$ and an edge between $x, y\in P$ if and only if $x$ and $y$ are incomparable. Let $H(P)$ be the graph entropy of the $\overline{G}(P)$. Cardinal, Fiorini, Joret, Jungers, and Monro \cite{cardinal2013sorting} showed that
\[|P|\cdot H(P)\le 2\log_2 e(P),\]
improving a result of Kahn and Kim \cite{kahn1995entropy}. This is the optimal constant; however, for width $2$ posets Fiorini and Rexhep \cite{fiorini2016entropy} showed that the multiplicative constant $2$ can be improved to $2-\epsilon$ for some $\epsilon\approx 0.26$ as long as $\overline{G}(P)$ has no connected component of size $2$. Their methods are unrelated to those of this work, but suggest that ``extremal posets'' with respect to statistics of linear extensions are somewhat separated from typical posets.

In Section~\ref{sec:path} we outline the key path-counting interpretation of linear extensions of width $2$ posets, and prove essential properties of the correspondence. In Section~\ref{sec:bound} we prove Theorem~\ref{thm:main} using those properties. In Section~\ref{sec:construction} we prove Theorem~\ref{thm:construction} by outlining the computations, also based on the path-counting interpretation of linear extensions, that determine $\delta(T_n)$. Auxiliary calculations for this section are contained in Appendix~\ref{app:calc}. Finally, we discuss the optimal constants and other outstanding questions in Section~\ref{sec:concluding}.

\section{Path-Counting Interpretation of Linear Extensions} \label{sec:path}
\subsection{The Interpretation}
Let $P$ be a finite, width $2$ poset. We can reinterpret linear extensions of $P$ in a natural way. First recall that a width $w$ poset can be partitioned into $w$ chains. Hence we can write $P = \{a_1, \ldots, a_m, b_1, \ldots, b_n\}$, where $a_1\le\cdots\le a_m$ and $b_1\le\cdots\le b_n$ (there may also be relations of the form $a_i\le b_j$ or $b_j\le a_i$).

\begin{definition}\label{def:grid-diagram}
The \emph{grid diagram} of $P$ is formed as follows. Draw an $m\times n$ grid. Label the segments along the bottom axis by $a_1, \ldots, a_m$ from left to right, and label the segments along the left axis by $b_1, \ldots, b_n$ from bottom to top. Let $C_{i, j}$ the cell in the $i$th row from the left and $j$th column from the bottom. We also label all $(m + 1)(n + 1)$ grid points so that the bottom left is $(0, 0)$ and the top right is $(m, n)$. Thus $(i, j)$ is the top right corner of $C_{i, j}$. Now, if $a_i\le b_j$ then color the cell $C_{i, j}$ red, while if $a_i\ge b_j$ then color the cell $C_{i, j}$ blue. Let $R$ be the set of red cells, and $B$ the set of blue cells.  Finally, let $S$ be the set of $C_{i, j}$ such that $1\le i\le m$, $1\le j\le n$, and $\mathbb{P}(a_i \prec b_j)\le\frac{1}{2}$.
\end{definition}
It will later be convenient to sometimes color cells $C_{0, i}$ for $1\le i\le n$ red and cells $C_{j, 0}$ for $1\le j\le m$ blue, but we do not include such cells in the definition of the sets $R$ and $B$.

An example of a width $2$ poset $P$ and its corresponding grid diagram is shown in Figure~\ref{fig:grid}. Notice also that the grid diagram of a poset may depend on its presentation $P = \{a_1, \ldots, a_m, b_1, \ldots, b_n\}$, since there may be multiple ways to decompose it into chains. This is a technicality which will not be too important.

In order the characterize the structure of grid diagrams, we recall that a Young diagram is a finite collection of boxes arranged in left-justified rows with decreasing row lengths from top to bottom. For our purposes, however, we will also allow a Young diagram to be right-justified with decreasing row lengths from bottom to top, i.e., rotated $180$ degrees.
\begin{proposition}\label{prop:br-young}
In the grid diagram of finite width $2$ poset $P$, the sets $R$ and $B$ form Young diagrams. $R$ is left- and top-justified, and $B$ is right- and bottom-justified. Furthermore, $R\cap B = \varnothing$.
\end{proposition}
\begin{proof}
We see that if $C_{i, j}$ is filled in red, then $a_i\le b_j$. Thus if additionally $i'\le i$ and $j\le j'$ then $a_{i'}\le a_i\le b_j\le b_{j'}$. In particular, $i'\le i$, $j\le j'$, and $(i, j)\in R$ together imply that $(i', j')\in R$. Therefore $R$ indeed corresponds to a Young diagram which is left- and top-justified. Similarly, $B$ corresponds to a Young diagram which is right- and bottom-justified. Clearly these two Young diagrams do not intersect.
\end{proof}

Any linear extension $\prec$ of $P$ can be written as a rearrangement of the symbols $a_1, \ldots, a_m, b_1, \ldots, b_n$, interpreted in increasing order. It must contain $a_1,\ldots, a_m$ in that order, and $b_1, \ldots, b_n$ in that order. Hence it corresponds to a path from the bottom left corner $(0, 0)$ to the top right corner $(m, n)$ of the grid diagram of $P$ that only goes up and right: the order of the segments used gives precisely the linear extension as written above.

Furthermore, it is easy to see that the paths corresponding to linear extensions of $P$ are precisely those that stay between $B$ and $R$.
\begin{definition}
A up-right path from $(0, 0)$ to $(m, n)$ in the grid diagram of $P$ is \emph{valid} if it stays between $B$ and $R$.
\end{definition}
Additionally, notice that $a_i\prec b_j$ if and only if the corresponding path goes below the cell $C_{i, j}$: the extension must make its $j$th up move after making its $i$th right move.

\begin{figure}[h]
\includegraphics{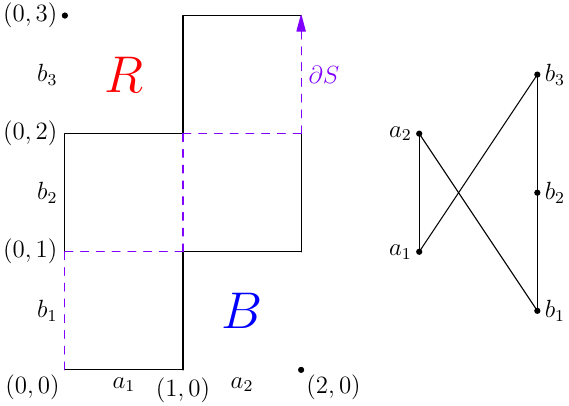}
\caption{Example of poset and grid diagram.}
\label{fig:grid}
\end{figure}

Now, we study the structure of $S$ using its definition from the end of Definition~\ref{def:grid-diagram}.
\begin{proposition}\label{prop:s-young}
In the grid diagram of finite width $2$ poset $P$, the set $S$ forms a right- and bottom-justified Young diagram satisfying $B\subseteq S$ and $S\cap R = \varnothing$.
\end{proposition}
\begin{proof}
Notice that if $a_{i'}\prec b_{j'}$, $i\le i'$, and $j'\le j$ then $a_i\prec b_j$. Thus
\[\mathbb{P}(a_{i'}\prec b_{j'})\le\mathbb{P}(a_i\prec b_j)\]
if $i\le i'$ and $j'\le j$. In particular, $i\le i'$, $j'\le j$, and $(i, j)\in S$ imply $(i', j')\in S$. Thus $S$ is a right- and bottom-justified Young diagram. Also, if $a_i\ge b_j$ then clearly $\mathbb{P}(a_i\prec b_j) = 0$ so that $(i, j)\in S$. Thus $B\subseteq S$. Similarly, we see that if $a_i\le b_j$ then $\mathbb{P}(a_i\prec b_j) = 1$, hence $S\cap R = \varnothing$.
\end{proof}

We will often consider the path $\partial S$ along the border of $S$, from the bottom left corner to top right corner of our grid diagram. From this path $S$ can be reconstructed. It is a path which stays between $R$ and $B$, by Proposition~\ref{prop:s-young}. An example is shown in Figure~\ref{fig:grid}. $\partial S$ will be denoted by a dotted arrow in all figures.

Incidentally, $\partial S$ thus corresponds to a linear extension $\prec_0$ of $P$ due to the correspondence between paths and linear extensions given earlier. Furthermore, we see that $a_i\prec_0 b_j$ if and only if $a_i$ comes before $b_j$ in $\partial S$, which means cell $C_{i,j}$ is not in $S$. By definition of $S$, this occurs precisely when $\mathbb{P}(a_i\prec b_j) > \frac{1}{2}$. This demonstrates that $x\prec_0 y$ if and only if $\mathbb{P}(x\prec y) > \frac{1}{2}$ or $\mathbb{P}(x\prec y) = \frac{1}{2}$ and $y = a_i$ and $x = b_j$ for some $i, j$. (Fishburn \cite{fishburn1976linear} showed that there are posets with no transitive relation $\prec_0$ satisfying $x\prec_0 y$ if $\mathbb{P}(x\prec y) > \frac{1}{2}$.)

\begin{figure}[h]
\includegraphics{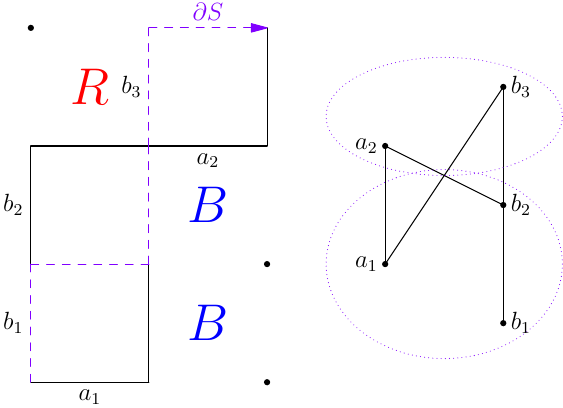}
\caption{Example of poset sum.}
\label{fig:sum}
\end{figure}

Finally, it will be useful to understand the structure of a grid diagram of a poset direct sum.
\begin{proposition}
If in the grid diagram of finite width $2$ poset $P$ the sets $B$ and $R$ have cells that share a vertex, then $P$ decomposes as a direct sum. Otherwise, if either the bottom left or top right cell of the grid diagram is in $B\cup R$, then $P$ decomposes as a direct sum.
\end{proposition}
\begin{proof}
First suppose that $B$ and $R$ have cells that share a vertex $(i, j)$. Consider the induced subposets $P_1$ and $P_2$ of $P$ obtained by restricting to $\{a_1, \ldots, a_i, b_1, \ldots, b_j\}$ and $\{a_{i + 1}, \ldots, a_m, b_{j + 1}, \ldots, b_n\}$, respectively.

Since $(i, j)$ is the vertex of some cell of $B$, we see that the cells $C_{i', j'}$ for $i'\ge i + 1$ and $j'\le j$ are in $B$. Similarly, the cells $C_{i', j'}$ for $i'\le i$ and $j'\ge j + 1$ are in $R$. Unwinding the definition of $B$ and $R$, this demonstrates that every element of $P_1$ is less than every element of $P_2$ when considered as part of the entire poset $P$. Thus $P = P_1\oplus P_2$, as desired.

Now suppose that the bottom left cell of the grid diagram is in $R$. The other three cases are symmetric. Then $C_{1, 1}\in R$, which means $C_{1, j}\in R$ for all $1\le j\le n$. Thus $a_1\le b_j$ for all $1\le j\le n$, while $a_1\le a_i$ for all $2\le i\le m$ by definition. Therefore every element of $\{a_1\}$ is less than every element of $\{a_2, \ldots, a_m, b_1, \ldots, b_m\}$, and a similar argument to above shows that $P$ decomposes as a direct sum.
\end{proof}
An example of direct sum decomposition is shown in Figure~\ref{fig:sum}.

\subsection{Path-Counting Inequalities}
Fix an underlying poset $P = \{a_1, \ldots, a_m, b_1, \ldots, b_n\}$ with $a_1\le\cdots\le a_m$ and $b_1\le\cdots\le b_n$ as earlier. We construct the grid diagram of $P$, defining $B$, $R$, and $S$ as above, and we additionally assume that $P$ does not decompose as a direct sum.

\begin{definition}\label{def:path}
Let $t_{i, j}$ be the number of up-right paths from $(0, 0)$ to $(i, j)$ that stay between $B$ and $R$. Let $r_{i, j}$ be the number of down-left paths from $(m, n)$ to $(i, j)$ that stay between $B$ and $R$.
\end{definition}
An example of this definition is shown in Figure~\ref{fig:count}. These numbers satisfy recursive relations $t_{i, j} = t_{i - 1, j} + t_{i, j - 1}$ and $r_{i, j} = r_{i + 1, j} + r_{i, j + 1}$ provided $(i, j)$ is connected to $(0, 0)$ via an up-right path that stays between $B$ and $R$. We define $t_{i, j} = r_{i, j} = 0$ unless $0\le i\le m$ and $0\le j\le n$. (Notice that the recursions do not hold for $(i, j)\in\{(2, 0), (0, 3)\}$ in Figure~\ref{fig:count}, and instead $t_{i, j} = r_{i, j} = 0$.)

\begin{figure}[h]
\includegraphics{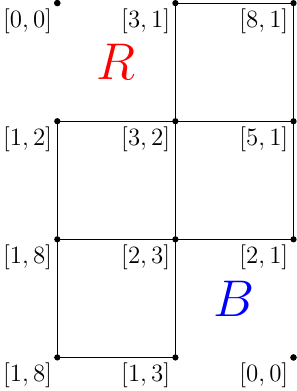}
\caption{Example of $[t_{i, j}, r_{i, j}]$ for the poset from Figure~\ref{fig:grid}.}
\label{fig:count}
\end{figure}

We will show that these sequences naturally give rise to log-concave sequences. Recall that a sequence $(c_i)_{i = 1}^k$ is \emph{log-concave} if $c_i > 0$ for $1\le i\le k$ and $c_i^2\ge c_{i - 1}c_{i + 1}$ for $2\le i\le k - 1$. We say that $(c_i)_{i = 1}^k$ is \emph{log-concave with surrounding zeros} if the zeros only form a block at the beginning and end of the sequence, and the remainder is log-concave.

First, we need a lemma about general log-concave sequences.
\begin{lemma}\label{lem:log-sum}
If $(c_i)_{i = 1}^k$ is log-concave then so is $(d_i)_{i = 1}^k$ where
\[d_i = \sum_{j = 1}^i c_j.\]
\end{lemma}
This is straightforward to prove with explicit computations and induction, but it is also a special case of a result of Hoggar \cite{hoggar1974chromatic} which states that the convolution of two log-concave sequences is log-concave. Convolution with a sequence of $1$s demonstrates the desired result.

\begin{figure}[h]
\includegraphics{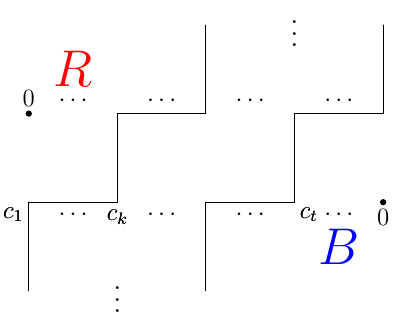}
\caption{Log-concavity proof figure.}
\label{fig:log-lemma}
\end{figure}

\begin{lemma}\label{lem:log-concavity}
For every $0\le j\le n$, the sequences $(t_{i, j})_{i = 0}^m$ and $(r_{i, j})_{i = 0}^m$ are log-concave with surrounding zeros. For every $0\le i\le m$, the sequences $(t_{i, j})_{j = 0}^n, (r_{i, j})_{j = 0}^n$ are log-concave with surrounding zeros.
\end{lemma}
\begin{proof}
Notice that no paths staying between $B$ and $R$ end at a point that is strictly within $B$ or $R$, so there are potentially many $0$s in these sequences, but we can see that they surround the positive part in the middle.

By symmetry it suffices to prove that $(t_{i, j})_{i = 0}^m$ is log-concave with surrounding zeros. We do this by induction on $j$.

The base case $j = 0$ is trivial: $t_{i, 0}\in\{0, 1\}$ for $0\le i\le m$, and the sequence starts with $t_{0, 0} = 1$ before permanently transitioning into $0$s after some point.

Now assume that $j = \ell + 1$, and assume the truth of the required assertion for $j = \ell$. Suppose that $(t_{i, \ell})_{i = 0}^m$ is of the form $0, 0, \ldots, 0, c_1, c_2, \ldots, c_t, 0, \ldots, 0$, where $c_i > 0$ for $1\le i\le t$. Then we know that $(c_i)_{i = 1}^t$ is log-concave.

Let $k$ be the first index such that the lattice point directly above $c_k$ is not in the strict interior of $R$. Figure~\ref{fig:log-lemma} depicts this situation, with the region between the solid lines denoting the cells not in $B\cup R$. We see that the next higher row of values, $(t_{i, \ell + 1})_{i = 0}^m$, is
\[0, \ldots, 0, c_k, c_k + c_{k + 1}, \ldots, c_k + \cdots + c_{t - 1}, c_k + \cdots + c_t, c_k + \cdots + c_t, \ldots, c_k + \cdots + c_t, 0, \ldots, 0,\]
using the recurrence $t_{i, j} = t_{i - 1, j} + t_{i, j - 1}$ in the obvious way. Now apply Lemma~\ref{lem:log-sum} and notice that $(c_k + \cdots + c_t)^2\ge (c_k + \cdots + c_{t - 1})(c_k + \cdots + c_t)$ to establish the log-concavity of the nonzero portion, which finishes the proof.
\end{proof}

Finally, it is useful to note that the number of valid paths through $(i, j)$ is $r_{i, j}t_{i, j}$.

\section{Bounding \texorpdfstring{$\delta(P)$}{delta(P)}} \label{sec:bound}
In this section we prove Theorem~\ref{thm:main}, which we restate here.
\thmmain*
\begin{proof}
Since $\delta(P\oplus Q) = \max\{\delta(P), \delta(Q)\}$, as noted in Section~\ref{sec:introduction}, we may assume that $P$ cannot be decomposed as a direct sum. By hypothesis, we know $P$ is not $\mathbf{1}$ or $\mathcal{E}$.

Now if the grid diagram of $P$ is a $1\times n$ or $m\times 1$ rectangle, then $B = R = \varnothing$ (else $P$ decomposes as a direct sum).

We then easily see that there is a pair $x, y\in P$ with $\mathbb{P}(x\prec y) = \frac{1}{2}$ if $n$ is odd, and if $n = 2k$ is even then there is a pair with $\mathbb{P}(x\prec y) = \frac{k}{2k + 1}\ge\frac{2}{5}$ when $k\ge 2$. Notice that $n = 2$ is impossible since $P$ is not $\mathcal{E}$.

Therefore, we can assume that $P$ is not decomposable as a direct sum, and has both dimensions at least $2$ in its grid diagram.

Now fix our underlying poset $P = \{a_1, \ldots, a_m, b_1, \ldots, b_n\}$ with $a_1\le\cdots\le a_m$ and $b_1\le\cdots\le b_n$ as in Section~\ref{sec:path}. We have $m, n\ge 2$ and know that $P$ does not decompose as a direct sum. We define $B$, $R$, and $S$ as before.

We will be doing casework on the configuration of $B$, $R$, and $S$ in the bottom left corner of the grid diagram of $P$. As noted earlier, we will often consider the path $\partial S$ along the border of $S$. In figures $\partial S$ is denoted by a dotted line where a solid grid line would normally be. The special property $\partial S$ has is that if a cell $C_{i, j}$ is below and to the right of $\partial S$, i.e. if $C_{i, j}\in S$, then at most $\frac{1}{2}$ of all valid paths pass below $C_{i, j}$. Thus, since $\delta(P)$ is the balance constant, we can deduce that actually at most a $\delta(P)$ fraction of valid paths pass below $C_{i, j}$. This property, which we call the \emph{balance property}, as well as its mirror for cells above and to the left of $\partial S$ will be exploited several times in the following argument.

\begin{figure}[h]
\includegraphics{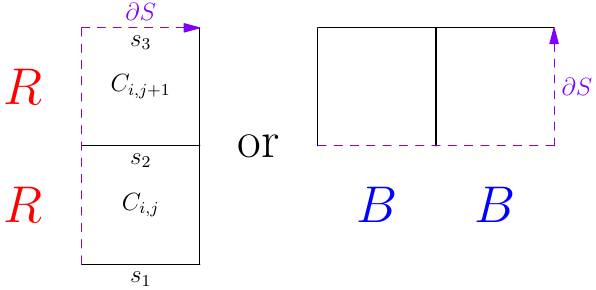}
\caption{Lemma~\ref{lem:twofifths} configurations.}
\label{fig:twofifths-lemma}
\end{figure}

For the statement of our first lemma, recall that the cells $C_{0, j}$ for $1\le j\le n$ are implicitly colored red and $C_{i, 0}$ for $1\le i\le m$ are implicitly colored blue; thus, the configurations identified in the lemma might be flush against the bottom or left of the grid diagram of $P$.
\begin{lemma}[$\frac{2}{5}$-Lemma]\label{lem:twofifths}
If one of the two images in Figure~\ref{fig:twofifths-lemma} appears within the grid diagram of $P$ with blue and red cells in the corresponding places and with $\partial S$ including the three segments shown, then
\[\delta(P)\ge\frac{2}{5}.\]
\end{lemma}
\begin{proof}
Let $\delta = \delta(P)$.

Without loss of generality we work with the situation on the left. Let $a$, $b$, and $c$ be the number of valid paths of the grid diagram of $P$ that pass through segments $s_1$, $s_2$, and $s_3$ respectively. Let $p = e(P)$ be the total number of valid paths, or equivalently linear extensions of $P$. Notice that the paths may also go through segments parallel to $s_1$, $s_2$, and $s_3$ that are outside of the small portion depicted in Figure~\ref{fig:twofifths-lemma}.

Now, since cell $C_{i, j + 1}$ is below $\partial S$, the balance property implies that the number of valid paths that pass below it is at most $\delta p$. Indeed, the fraction of valid paths passing on segments below it is precisely the probability $\mathbb{P}(a_i\prec b_{j + 1})$, and since the fraction is at most $\frac{1}{2}$ (since $C_{i,j+1}\in S$), it must actually be at most $\delta$. Similarly, the number of valid paths passing this horizontal section strictly above segment $s_3$ is at most $\delta p$. For this we use the balance property applied to cell $C_{i, j + 2}$. Explicitly, the number of such paths is precisely the probability $\mathbb{P}(b_{j+2}\prec a_i)$. It is at most $\frac{1}{2}$ (since $C_{i,j+2}\notin S$) hence it must actually be at most $\delta$. Also, if $C_{i,j+2}$ does not exist, then we are at the edge of the grid and this number of paths above is in fact $0$. Now, the total number of valid paths is $p$, hence the number of valid paths passing through segment $s_3$ must be at least $(1 - 2\delta)p$. That is, $c\ge (1 - 2\delta)p$.

However, it is also clear that $a\ge b\ge c$. Indeed, we can exhibit an injection from valid paths passing through $s_3$ to valid paths passing through $s_2$ (and similar for $s_2$ and $s_1$): any such path has a portion from $(i - 1, j)$ to $(i, j + 1)$ that we reflect over the cell $C_{i, j + 1}$.

Thus $a, b\ge (1 - 2\delta)p$. However, we showed earlier that the number of paths passing below cell $C_{i, j + 1}$ was at most $\delta p$. This yields
\[2(1 - 2\delta)p\le a + b\le\delta p,\]
which gives the desired result as $p > 0$.
\end{proof}

Now suppose that $P$ also satisfies $\delta(P) < \lambda = \frac{-3 + 5\sqrt{17}}{52}$. We shall derive a contradiction using Lemma~\ref{lem:twofifths} and some cases. Without loss of generality, $\partial S$, which starts at $(0, 0)$, goes through $(1, 0)$. Indeed, if not then we can switch the symmetric roles of $\{a_1,\ldots,a_m\}$ and $\{b_1,\ldots,b_n\}$, which serves to reflect the grid diagram.

\begin{lemma}[Structure Lemma]\label{lem:structure}
The bottom left of the grid diagram of $P$ is one of the $9$ diagrams depicted in Figure~\ref{fig:structure-lemma}.
\end{lemma}
\begin{proof}
Recall that no blue cells are above $\partial S$ and no red cells are below it. Notice that since $P$ does not decompose as a direct sum, $C_{1, 1}$ is neither colored red nor blue.

We know $\delta(P) < \frac{2}{5}$, which means that by Lemma~\ref{lem:twofifths} the path $\partial S$ must go through $(1, 1)$. Notice that $\partial S$ must continue past this point since we know that the grid diagram has both dimensions at least $2$.

If $\partial S$ continues through $(1, 2)$, then consider $C_{1, 2}$. Either it is empty, yielding Case~\ref{case:1}, or it is red. If it is red, we can again apply Lemma~\ref{lem:twofifths} to deduce that $\partial S$ must then immediately go right, passing through $(2, 2)$. Indeed, if it goes farther up and then later goes right, then a figure as in Lemma~\ref{lem:twofifths} will appear. Recalling that $P$ is not a direct sum, the cells $C_{2, 1}$ and $C_{2, 2}$ are not blue and thus we are in Case~\ref{case:2}.

Otherwise $\partial S$ goes through $(2, 1)$. If $C_{2, 1}$ is blue we get Case~\ref{case:3}. Otherwise it has no color.

Now we look at where $\partial S$ continues. If it goes through $(3, 1)$ next then there are two possibilities. If $C_{3, 1}$ is not blue, then we obtain Case~\ref{case:4}. Otherwise, we see by Lemma~\ref{lem:twofifths} that $\partial S$ must immediately go up. Now we see that $C_{3, 2}$ and $C_{2, 2}$ should not be red since otherwise this would violate the condition that $P$ does not decompose as a direct sum. Thus they have no color: as they are above $\partial S$, they cannot be blue. Thus there are only two cases, depending on whether or not $C_{1, 2}$ is red. Not red gives Case~\ref{case:5} and red gives Case~\ref{case:6}.

The other possibility is that $\partial S$ continues through $(2, 2)$ instead. Now we look at $C_{1, 2}$. If it is empty, we obtain Case~\ref{case:7}. Otherwise, it is red. Now, look at $C_{2, 2}$. If it is empty, we obtain Case~\ref{case:9}.

Finally, we have the case in which both $C_{1, 2}$ and $C_{2, 2}$ are red. Then, by similar arguments using Lemma~\ref{lem:twofifths}, $\partial S$ must immediately move right. Then we see by virtue of $P$ not decomposing as a direct sum that $C_{3, 1}$ and $C_{3, 2}$ are empty. Thus we are left with Case~\ref{case:8}.
\end{proof}

\begin{figure}[h]
\includegraphics{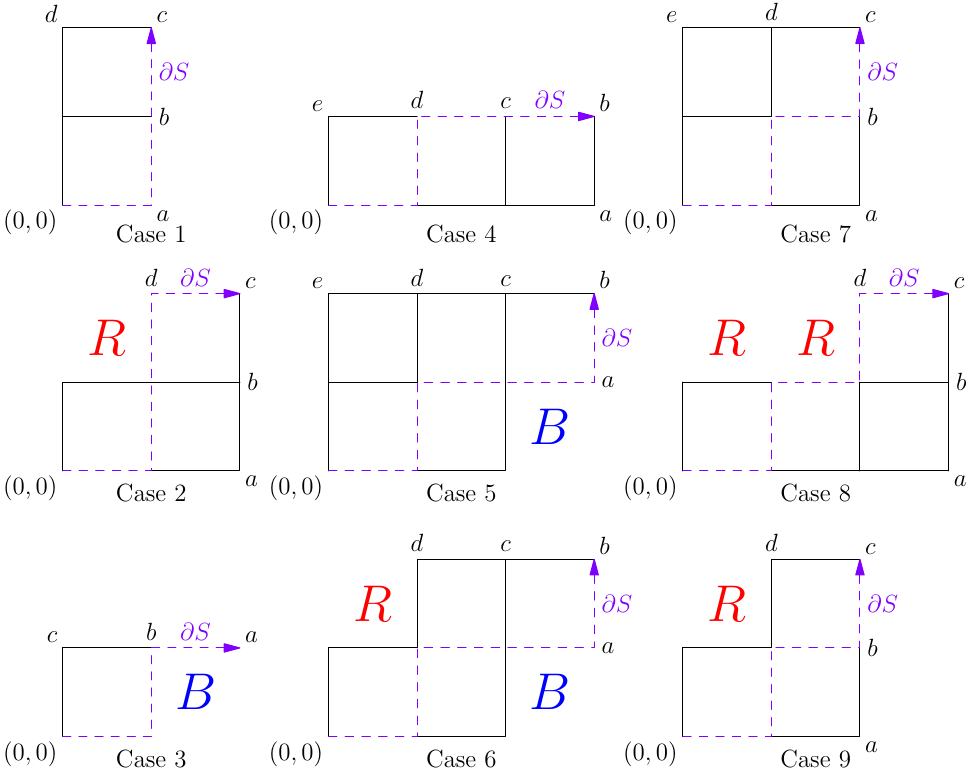}
\caption{Cases of Lemma~\ref{lem:structure}.}
\label{fig:structure-lemma}
\end{figure}

Now we dispatch the $9$ cases in order. Many of the proofs only need to use linear inequalities, but Cases $4$ and $9$ need the log-concavity inequalities of Lemma~\ref{lem:log-concavity}.

Again, let $p = e(P) = r_{0, 0} = t_{m, n}$ be the number of valid paths. All the cases are similar to the first (although Cases $4$ and $9$ have some modifications), so they are condensed.

\begin{case}\label{case:1}
\end{case}
\begin{proof}
Let $(a, b, c, d) = \frac{1}{p}(r_{1, 0}, r_{1, 1}, r_{1, 2}, r_{0, 2})$, as indicated in Figure~\ref{fig:structure-lemma}. We then easily see that the fraction of valid paths that pass through $(1, 0)$ is precisely $a$. Similarly, the fraction of paths going through $(1, 1)$ is $2b$, the fraction going through $(1, 2)$ is $3c$, and the fraction going through $(0, 2)$ is $d$. (We are using the fact from Section~\ref{sec:path} that the number of valid paths through $(i, j)$ is $t_{i, j}r_{i, j}$.)

Since cell $C_{1, 1}$ is above $\partial S$, the balance property tells us that the fraction of valid paths that go above $C_{1, 1}$ is at most $\delta$. We can see that this fraction is precisely $\frac{1\cdot r_{0, 1}}{p} = b + d$, using the recurrence for the $r_{i, j}$ sequence. Thus $b + d\le\delta$.

The cell $C_{2, 2}$ is below $\delta S$. Thus the fraction of paths above it is at least $1 - \delta$ (if this cell does not exist in the grid diagram, then the fraction is exactly $1$, which also satisfies this inequality). These paths go through the segment from $(0, 1)$ to $(0, 2)$ or the segment from $(1, 1)$ to $(1, 2)$. The total number of such paths is $2(cp) + 1(dp) = (2c + d)p$. Thus the fraction is $2c + d$, yielding $2c + d\ge 1 - \delta$.

Additionally, it is clear from the recurrence relation that $a\ge b\ge c\ge 0$ and $d\ge c\ge 0$. Finally, $p = r_{0, 0} = r_{0, 1} + r_{1, 0} = r_{0, 1} + r_{1, 1} + r_{0, 2} = (a + b + d)p$, hence $a + b + d = 1$.

Thus, overall, we have
\begin{align*}
\delta&\ge b + d,\\
2c + d&\ge 1 - \delta,\\
a&\ge b,\\
b&\ge c,\\
d&\ge c,\\
c&\ge 0,\\
a + b + d &= 1.\\
\end{align*}
In this linear programming relaxation, we can show that $\delta\ge\frac{2}{5}$. Indeed, multiply the above relations respectively by $\frac{3}{5}, \frac{2}{5}, 0, \frac{3}{5}, \frac{1}{5}, 0, 0$ and add.

This contradicts our assumption that $\delta(P) < \lambda = \frac{-3 + 5\sqrt{17}}{52}$.
\end{proof}
\begin{case}\label{case:2}
\end{case}
\begin{proof}
Let $(a, b, c, d) = \frac{1}{p}(r_{2, 0}, r_{2, 1}, r_{2, 2}, r_{1, 2})$. Similar to Case~\ref{case:1}, we see $a\ge b\ge c\ge 0$ and $d\ge c$.

Using that $C_{1, 1}$ is above $\partial S$, we find that the fraction of valid paths going above it is at most $\delta$. Thus $b + d\le\delta$, since we easily see $r_{0, 1} = r_{1, 1} = r_{2, 1} + r_{1, 2} = (b + d)p$.

Using that $C_{2, 2}$ is below $\partial S$, the fraction of valid paths going above it is at least $1 - \delta$. Hence we see $2d\ge 1 - \delta$.

Finally, we use $C_{2, 3}$. The fraction of valid paths that go below it is at least $1 - \delta$. (Again, if $C_{2, 3}$ is not in the grid diagram since the diagram only has two columns, then the fraction is in fact $1$.) The fraction of such paths is $a + 2b + 2c\ge 1 - \delta$.

Hence
\begin{align*}
\delta&\ge b + d,\\
2d&\ge 1 - \delta,\\
a + 2b + 2c&\ge 1 - \delta,\\
a&\ge b,\\
b&\ge c,\\
d&\ge c,\\
c&\ge 0,\\
a + 2b + 2d &= 1.\\
\end{align*}
Multiplying these relations respectively by $\frac{2}{5}, \frac{2}{5}, \frac{1}{5}, 0, \frac{2}{5}, 0, 0, -\frac{1}{5}$ and adding yields $\delta\ge\frac{2}{5}$. (Note that the last relation is an equality, hence the negative coefficient is justified.)

This again contradicts our assumption.
\end{proof}
\begin{case}\label{case:3}
\end{case}
\begin{proof}
We see $0\le a\le b\le c$, and using $C_{1, 1}$ gives $b\ge 1 - \delta$. We also have $b + c = 1$. Thus $1 = b + c\ge 2b\ge 2(1 - \delta)$, hence $\delta\ge\frac{1}{2}$, contradicting our assumption.
\end{proof}
\begin{case}\label{case:4}
\end{case}
\begin{proof}
Again $a\ge b\ge 0, e\ge d\ge c\ge b$, and $a + c + d + e = \frac{r_{0, 0}}{p} = 1$. Using that $C_{1, 1}$ is above $\partial S$ yields $e\le\delta$; using $C_{2, 1}$ gives $a + c\le\delta$; using $C_{3, 2}$ yields $a + 3b\ge 1 - \delta$.

Finally, we (for the first time) need Lemma~\ref{lem:log-concavity} on the log-concavity of the $t$ and $r$ sequences. In particular, it implies that $r_{2, 1}^2\ge r_{1, 1}r_{3, 1}$, which after dividing by $p$ yields $c^2\ge bd$. Thus
\begin{align*}
\delta &\ge e,\\
\delta &\ge a + c,\\
a + 3b + \delta &\ge 1,\\
c^2 &\ge bd,\\
a\ge b&\ge 0,\\
e\ge d\ge c &\ge b,\\
a + c + d + e &\ge 1.
\end{align*}
(We have weakened the equality to an inequality.) We claim that any real numbers satisfying these inequalities will also satisfy $\delta > \frac{9}{25} = 0.36$, which will finish this case. Assume for the sake of contradiction that there is some solution with $\delta\le 0.36$. Notice that multiplying all of $(\delta, a, b, c, d, e)$ by $\mu > 1$ preserves all of these inequalities, hence we may assume that $\delta = 0.36$.

Now $a + c\le 0.36$ and $a + c + d + e\ge 1$ so $d + e\ge 0.64$. Also, $0.36\ge e$, so $d\ge 0.28$. Then $c^2\ge bd\ge 0.28b$.

We have $a\ge b$ and $a\ge 0.64 - 3b$ while $a\le 0.36 - c$, giving $0.36 - c\ge b$ and $0.36 - c\ge 0.64 - 3b$. Thus $b + c\le 0.36$ and $3b - c\ge 0.28$. Recall that $c\ge b$.

Thus $c\le 0.36 - b$ and $c\le 3b - 0.28$, yielding $b\le 3b - 0.28$ and thus $b\ge 0.14$ as well as $b\le 0.36 - b$ and thus $b\le 0.18$. Therefore $b\in [0.14, 0.18]$, which contradicts $0.28b\le c^2\le (3b - 0.28)^2$. Thus we are done with this case.

(As it happens, the optimal constant for the system above is $\delta = \frac{-1 + 2\sqrt{13}}{17}$.)
\end{proof}
\begin{case}\label{case:5}
\end{case}
\begin{proof}
Again $a\ge b\ge 0, e\ge d\ge c\ge b$, and $3a + 3c + 2d + e = 1$. Using $C_{1, 1}$ yields $2a + 2c + d\ge 1 - \delta$. Using $C_{2, 1}$ yields $a + c\le\delta$. Using $C_{3, 2}$ yields $3a\ge 1 - \delta$. Finally, using $C_{4, 2}$ gives $3b + 3c + 2d + e\ge 1 - \delta$. It is easily checked that this linear program yields $\delta\ge\frac{2}{5}$, giving the desired contradiction.
\end{proof}
\begin{case}\label{case:6}
\end{case}
\begin{proof}
Again $a\ge b\ge 0$, and $d\ge c\ge b$, and $3a + 3c + 2d = 1$. Using $C_{1, 1}$ yields $2a + 2c + d\ge 1 - \delta$. Using $C_{2, 1}$ yields $a + c\le\delta$. Using $C_{3, 2}$ yields $3a\ge 1 - \delta$. Finally, using $C_{4, 2}$ yields $3b + 3c + 2d\ge 1 - \delta$. (These are essentially the same as last case, except $e = 0$ and we do not have $e\ge d$.)

It is easily checked that this linear program gives $\delta\ge\frac{5}{13}$, giving the desired contradiction.
\end{proof}
\begin{case}\label{case:7}
\end{case}
\begin{proof}
Again $a\ge b\ge c\ge 0$, and $e\ge d\ge c$, and $a + 2b + 2d + e = 1$. Using $C_{1, 1}$ yields $b + d + e\le\delta$; using $C_{2, 1}$ yields $a\le\delta$; using $C_{2, 2}$ yields $a + 2b\ge 1 - \delta$; using $C_{3, 2}$ yields $3c + 2d + e\ge 1 - \delta$. It is easily checked that this linear program yields $\delta\ge\frac{7}{19}$, giving the desired contradiction.
\end{proof}
\begin{case}\label{case:8}
\end{case}
\begin{proof}
Again $a\ge b\ge c\ge 0$, and $d\ge c$, and $a + 3b + 3d = 1$. Using $C_{2, 1}$ yields $a + b + d\le\delta$; using $C_{3, 2}$ yields $a + 3b\le\delta$; using $C_{3, 3}$ yields $a + 3b + 3c\ge 1 - \delta$. It is easily checked that this linear program yields $\delta\ge\frac{9}{23}$, giving the desired contradiction.
\end{proof}
\begin{case}\label{case:9}
\end{case}
\begin{proof}
Again $a\ge b\ge c\ge 0$, and $d\ge c$, and $a + 2b + 2d = 1$. Using $C_{1, 1}$ yields $b + d\le\delta$; using $C_{2, 1}$ yields $a\le\delta$; using $C_{2, 2}$ yields $2d\le\delta$; using $C_{3, 2}$ yields $3c + 2d\ge 1 - \delta$. Finally, the log-concavity result of Lemma~\ref{lem:log-concavity} gives $b^2\ge ac$, similar to before. Thus, in particular we have
\begin{align*}
\delta &\ge b + d,\\
\delta &\ge a,\\
\delta &\ge 2d,\\
3c + 2d + \delta &\ge 1,\\
b^2&\ge ac,\\
a\ge b\ge c &\ge 0,\\
d &\ge c,\\
a + 2b + 2d &\ge 1.
\end{align*}
As in Case~\ref{case:4}, we weakened the equality to an inequality. Now assume for the sake of contradiction that there was a solution $(\delta, a, b, c, d)$ to these inequalities with $\delta < \lambda = \frac{-3 + 5\sqrt{17}}{52}$. Then multiply each of $(\delta, a, b, c, d)$ by $\mu > 1$ to make $\delta = \lambda$. All the inequalities are preserved, except that $3c + 2d + \delta > 1$ and $a + 2b + 2d > 1$ are now strict.

Now $b + d\le\lambda$ and $a + 2b + 2d > 1$ give $a > 1 - 2\lambda$, hence $b^2 > (1 - 2\lambda)c$ (or $b = c = 0$, which immediately implies $2\lambda\ge\lambda + 2d > 1$, a contradiction). Additionally, the inequalities $a\le\lambda$ and $a + 2b + 2d > 1$ give $b + d > \frac{1 - \lambda}{2}$. Thus we have the following list of inequalities:
\begin{align*}
b + d\le\lambda, 2d\le\lambda, c\le d,\\
3c + 2d > 1 - \lambda, b^2 > (1 - 2\lambda)c, 2b + 2d > 1 - \lambda.
\end{align*}
In particular, we have $3c > 1 - \lambda - 2d$ along with $c\le d$ and $c < \frac{b^2}{1 - 2\lambda}$. Thus $5d > 1 - \lambda$ and $3b^2 > (1 - 2\lambda)(1 - \lambda - 2d)$. This gives us
\begin{align*}
b + d\le\lambda, 2d\le\lambda,\\
5d > 1 - \lambda, 2b + 2d > 1 - \lambda, 3b^2 > (1 - 2\lambda)(1 - 2\lambda - 2d).
\end{align*}
Now let $x = b + d$ and $y = d$. Then we know $x\in\left(\frac{1 - \lambda}{2}, \lambda\right]$ and $y\in\left(\frac{1 - \lambda}{5}, \frac{\lambda}{2}\right]$, as well as
\[3(x - y)^2 - (1 - 2\lambda)(1 - \lambda - 2y) > 0.\]
Since $\lambda < \frac{1}{2}$ we know that $x > y$, so that $\lambda - y\ge x - y > 0$. Thus $3(\lambda - y)^2 - (1 - 2\lambda)(1 - \lambda - 2y) > 0$. Since $\lambda = \frac{-3 + 5\sqrt{17}}{52}$, we can check that the roots of this quadratic are $\frac{1 - \lambda}{5} < \frac{53\lambda - 13}{15}$. Thus we have $y < \frac{1 - \lambda}{5}$ or $y > \frac{53\lambda - 13}{15} > \frac{\lambda}{2}$. But this contradicts our assertion that $y\in\left(\frac{1 - \lambda}{5}, \frac{\lambda}{2}\right]$. We have our desired contradiction.
\end{proof}

Thus all cases are exhausted, and the theorem is proved.
\end{proof}

\section{A Sequence of Posets with Small Balance Constant}\label{sec:construction}
We now construct the posets $T_n$ and give the major calculations demonstrating that
\[\delta(T_n)\rightarrow\beta = \frac{5864893 + 27\sqrt{57}}{16812976},\]
thus proving Theorem~\ref{thm:construction}. Many of the minor calculations have been relegated to Appendix~\ref{app:calc}.

The poset $T_n$ depends on a positive integer parameter $n$, and has Hasse diagram as in Figure~\ref{fig:constructed-poset}. There are $n$ copies of the circled pattern in the middle section; there are five initial circled patterns and five terminal circled patterns. Notice that the top strand has $2n + 21$ elements, and the bottom has $2n + 20$.

\begin{figure}
\includegraphics{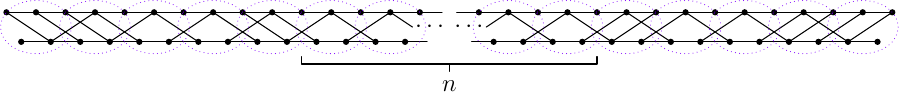}
\caption{Hasse diagram of $T_n$, elements increasing from left to right.}
\label{fig:constructed-poset}
\end{figure}

\begin{figure}
\includegraphics{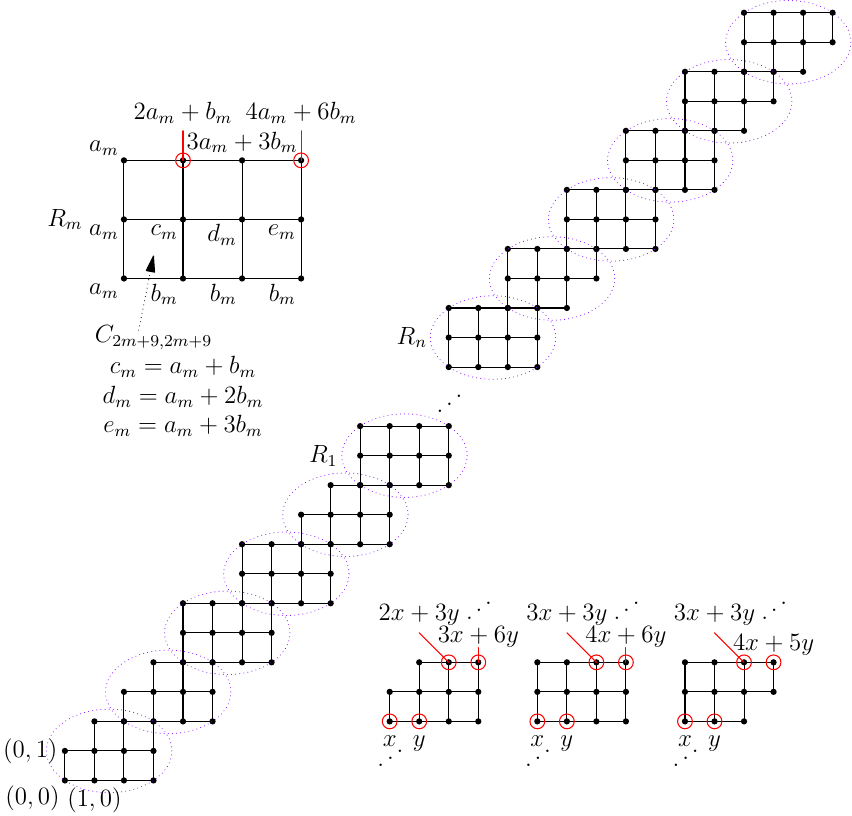}
\caption{Grid diagram of $T_n$, with $t_{i, j}$ relations.}
\label{fig:constructed-diagram}
\end{figure}

The grid diagram is shown in Figure~\ref{fig:constructed-diagram} with $t_{i, j}$ numbers. It is made to be wider than it is tall by $1$. Notice that the $n$ repeated objects from the Hasse diagram correspond to the $n$ different $3\times 2$ rectangles, which are denoted by $R_1, \ldots, R_n$. We let $a_m = t_{2m + 8, 2m + 8}$ and $b_m = t_{2m + 9, 2m + 8}$ for $1\le m\le n + 1$. Thus if $1\le m\le n$ then $a_m$ and $b_m$ are the values of $t_{i, j}$ in the bottom left of $R_m$, as shown in Figure~\ref{fig:constructed-diagram}. Similarly, if $1\le m\le n$ then $a_{m + 1}$ and $b_{m + 1}$ are the values of $t_{i, j}$ in the top right of $R_m$.

This allows us to (using the $t_{i, j}$ recurrences) determine the values at every point in the grid. Notice that $a_{m + 1} = 3a_m + 3b_n$ and $b_{m + 1} = 4a_m + 6b_m$ since the pair $(a_{m + 1}, b_{m + 1})$ corresponds to the top right portion of $R_m$.

We can also explicitly compute $t_{i, j}$ for $0\le i\le 11$ and $0\le j\le 10$, although the results are not pictured here. Such computation yields $(a_1, b_1) = (19212, 35784)$. This allows us to solve the linear recurrences for $(a_m, b_m)$, hence determining $t_{i, j}$ values within $R_1, \ldots, R_n$. Additionally, with the pair $(a_{n + 1}, b_{n + 1})$ we can determine the $t_{i, j}$ values for the top right portion of the grid diagram of $T_n$. In particular, we can check that $t_{2n + 21, 2n + 20} = p = 16572a_{n + 1} + 19212b_{n + 1}$.

Notice that the grid diagram is symmetric under $180$ degree rotation, which means that $r_{i, j} = t_{2n + 21 - i, 2n + 20 - j}$. Now, the total number of paths through $(i, j)$ is $r_{i, j}t_{i, j} = t_{2n + 21 - i, 2n + 20 - j}t_{i, j}$. For a given cell $C_{i, j}$, we want to calculate the fraction of paths that go below or above it. Notice that for $C_{1, 1}$, the number of paths above it equals the number of paths through $(0, 1)$, which is $r_{0, 1} = t_{2n + 21, 2n + 19} = 5781a_{n + 1} + 6702b_{n + 1}$. Hence the fraction is
\begin{align*}
\mathbb{P}(y\prec x) = \frac{5781a_{n + 1} + 6702b_{n + 1}}{16572a_{n + 1} + 19212b_{n + 1}},
\end{align*}
where $y$ is the smallest of the chain of length $2n + 20$ while $x$ is the smallest of the chain of length $2n + 21$.

Solving the recurrence shows that $a_n$ and $b_n$ are combinations of $(9\pm\sqrt{57})^n$, and that the above fraction limits to $\beta = \frac{5864893+27\sqrt{57}}{16812976}$ as $n\rightarrow\infty$.

We claim that the above probability is the closest to $\frac{1}{2}$, i.e., it equals the balance constant $\delta(T_n)$. This completes the proof.

We have complete control over the number of paths through any square: for example, the number of paths through the bottom left corner of $R_m$ is $a_m\cdot (4a_{n + 1 - m} + 6b_{n + 1 - m})$. Furthermore, most cells $C_{i, j}$ are already blue or red (which corresponds to $\mathbb{P}(x\prec y)$ values of $0, 1$). The ones that are not fall into finitely many classes: the ones at the two ends, and ones within $R_m$. Each of these have fraction explicitly computable as a fraction of these recurrent sequences. For instance, $\frac{a_m(4a_{n + 1 - m} + 6b_{n + 1 - m})}{16572a_{n + 1} + 19212b_{n + 1}}$ is one such ratio. Thus the matter at hand reduces to finitely many inequalities of linearly recurrent sequences (for which we have explicit formulas). Checking the details is not so interesting. Refer to Appendix~\ref{app:calc} for these details.

This justifies Theorem~\ref{thm:construction}.

\section{Further remarks} \label{sec:concluding}

\subsection{Optimal Constants}\label{sub:opt}
Using computers to extend the casework in the method we use a little further seems to indicate that, in fact, for width $2$ posets not obtainable from $\mathbf{1}$ and $\mathcal{E}$ using direct sums, we have $\delta(P)\ge 0.348842$ or so. However, due to numerical precision issues it is not stated as a result here. Nevertheless, we expect it to not be hard to use quantifier elimination programs to verify the constant to this level of precision. There seem to be fundamental obstructions to further exploring the tree of cases, however: at some point it seems to be impossible to effectively prune the branches of the tree while approaching the optimal constant.

\subsection{Conjectures and Further Questions}
Regardless of the difficulties mentioned above, we ask whether the family exhibited above is optimal.
\begin{conjecture}
If $P$ is a finite, width $2$ poset that cannot be formed from $\mathbf{1}$ and $\mathcal{E}$ using the operation of direct sum, then
\[\delta(P)\ge\beta.\]
\end{conjecture}
Numerical results of Peczarski \cite{peczarski2017worst} on small posets suggest that the optimal constant is near $0.348843$. Combining with the speculations above, this suggests that $\beta$ is in roughly the right range to be the optimal constant. We can also ask whether our results extend to \emph{all} posets, not just width $2$.
\begin{conjecture}
There exists an absolute constant $\varepsilon > 0$ such that if $P$ is a finite poset not obtainable from $\mathbf{1}$ and $\mathcal{E}$ using direct sums, then
\[\delta(P)\ge\frac{1}{3} + \varepsilon.\]
(Can we take $\varepsilon = \beta - \frac{1}{3}$?)
\end{conjecture}
Olson and Sagan \cite{olson2018on} asked if, given any poset $P$ of width at least $3$, there exists a poset $Q$ with smaller width such that $\delta(Q) < \delta(P)$. We ask a more general and precise question.
\begin{conjecture}
Let $\mathcal{B}_w = \{\delta(P): P\text{ is a finite poset of width }w\}$ and let $\delta_w = \inf\mathcal{B}_w$. Then
\[0 = \delta_1 < \frac{1}{3} = \delta_2 < \delta_3 < \cdots.\]
\end{conjecture}
Kahn and Saks \cite{kahn1984balancing} pose the following conjecture.
\begin{conjecture}
\[\lim_{w\rightarrow\infty} \delta_w = \frac{1}{2}.\]
\end{conjecture}
Other questions about balance constants of partial orders can be found in the survey of Brightwell \cite{brightwell1999balanced}. We conclude by asking the following.
\begin{question}
What, in general, can be said about the topologies of $\mathcal{B}$ and $\mathcal{B}_w$? What can be said of the structure of the fibers $\delta^{-1}(r)$ for $r\in\mathbb{Q}\cap\left[0, \frac{1}{2}\right]$?
\end{question}

\section*{Acknowledgements}
This research was conducted at the University of Minnesota, Duluth REU run by Joe Gallian. It was supported by NSF/DMS grant 1659047 and NSA grant H98230-18-1-0010. The author would like to thank Joe Gallian for running the program and for useful comments. The author would also like to thank Mitchell Lee and Evan Chen for helpful comments on the manuscript. Finally, the author thanks the anonymous referees for several helpful comments which improved the paper.

\appendix
\section{Construction Computations}\label{app:calc}
We use the notations and facts introduced in Section~\ref{sec:construction}. First we will study and compute the values of $t_{i, j}$ for all relevant points $(i, j)$. Then, using the fact that the number of valid paths passing through $(i, j)$ is $t_{i, j}r_{i, j} = t_{2n + 21 - i, 2n + 20 - j}t_{i, j}$ as noted in Section~\ref{sec:construction}, we compute the number of valid paths going above and below each cell $C_{i, j}$. We then relate these fractions to the balance constant and prove the required inequalities. In the end there will be $27$ inequalities in the positive integer $n$ and $3$ inequalities in the positive integers $1\le m\le n$.

For convenience, we will denote $p_{i, j} = t_{i, j}r_{i, j}$ and $p = p_{0, 0}$. Note that $p$ is the number of linear extensions of $T_n$.

\subsection{Computing \texorpdfstring{$t_{i, j}$}{tij}}
Using $t_{0, 0} = 1$ and the recurrence for $t_{i, j}$ from Section~\ref{sec:path} allows us to explicitly compute $t_{i, j}$ for $0\le i\le 11$ and $0\le j\le 10$. They are compiled in Figure~\ref{fig:tij-top} (note that these must be rotated $90$ degrees counterclockwise if one wants to superimpose it over the grid diagram Figure~\ref{fig:constructed-diagram}).

Notice that, as mentioned in Section~\ref{sec:construction}, this means $(a_1, b_1) = (19212, 35784)$. Now, we also want to compile the values of $t_{2n + 21 - i, 2n + 20 - j}$ for $0\le i\le 11$ and $0\le j\le 10$. We can use the recurrence of Section~\ref{sec:path} along with the initial values $t_{2n + 10, 2n + 10} = a_{n + 1}$ and $t_{2n + 11, 2n + 10} = b_{n + 1}$ to compute these as linear combinations of $a_{n + 1}$ and $b_{n + 1}$.

Thus we can write $t_{2n + 21 - i, 2n + 20 - j} = c_{i, j}a_{n + 1} + d_{i, j}b_{n + 1}$ for positive integers $c_{i, j}$ and $d_{i, j}$ when $0\le i\le 11$ and $0\le j\le 10$. These are tabulated in Figure~\ref{fig:tij-bot-a} and Figure~\ref{fig:tij-bot-b}.

Putting it all together, we see that the number of valid paths $p_{i, j}$ through $(i, j)$ when $0\le i\le 11$ and $0\le j\le 10$ is precisely $t_{i, j}t_{2n + 21 - i, 2n + 20 - j} = (t_{i, j}c_{i, j})a_{n + 1} + (t_{i, j}d_{i, j})b_{n + 1}$. The total number of valid paths is this expression evaluated at $(i, j) = (0, 0)$, which is $p = 16572a_{n + 1} + 19212b_{n + 1}$.

Recall that the fraction of valid paths going above $C_{1, 1}$ was
\[\frac{p_{0, 1}}{p} = \frac{5781a_{n + 1} + 6702b_{n + 1}}{16572a_{n + 1} + 19212b_{n + 1}} < \frac{1}{2},\]
which we claimed to be the balance constant $\delta(T_n)$. Thus we must show that every other fraction is either greater than $\frac{1}{2}$ or at most this quantity.

\begin{figure}[h]
    \centering
    \begin{tabular}{c|*{11}{r}}
        \diaghead(1,-1){\hskip 0.03\hsize}{$i$}{$j$}
           &  0 &  1 &  2 &  3 &  4 &   5 &   6 &    7 &    8 &     9 &    10 \\ \hline
         0 &  1 &  1 &  0 &  0 &  0 &   0 &   0 &    0 &    0 &     0 &     0 \\
         1 &  1 &  2 &  2 &  0 &  0 &   0 &   0 &    0 &    0 &     0 &     0 \\
         2 &  1 &  3 &  5 &  5 &  0 &   0 &   0 &    0 &    0 &     0 &     0 \\
         3 &  1 &  4 &  9 & 14 & 14 &   0 &   0 &    0 &    0 &     0 &     0 \\
         4 &  0 &  0 &  9 & 23 & 37 &  37 &  37 &    0 &    0 &     0 &     0 \\
         5 &  0 &  0 &  9 & 32 & 69 & 106 & 143 &    0 &    0 &     0 &     0 \\
         6 &  0 &  0 &  0 &  0 & 69 & 175 & 318 &  318 &  318 &     0 &     0 \\
         7 &  0 &  0 &  0 &  0 & 69 & 244 & 562 &  880 & 1198 &     0 &     0 \\
         8 &  0 &  0 &  0 &  0 &  0 &   0 & 562 & 1442 & 2640 &  2640 &     0 \\
         9 &  0 &  0 &  0 &  0 &  0 &   0 & 562 & 2004 & 4644 &  7284 &  7284 \\
        10 &  0 &  0 &  0 &  0 &  0 &   0 &   0 &    0 & 4644 & 11928 & 19212 \\
        11 &  0 &  0 &  0 &  0 &  0 &   0 &   0 &    0 & 4644 & 16572 & 35784 \\
    \end{tabular}
    \caption{Values of $t_{i, j}$ for $0\le i\le 11$ and $0\le j\le 10$.}
    \label{fig:tij-top}
\end{figure}
\begin{figure}[h]
    \centering
    \begin{tabular}{c|*{11}{r}}
        \diaghead(1,-1){\hskip 0.03\hsize}{$i$}{$j$}
           &     0 &    1 &    2 &   3 &   4 &   5 &  6 &  7 & 8 & 9 & 10 \\ \hline
         0 & 16572 & 5781 &    0 &   0 &   0 &   0 &  0 &  0 & 0 & 0 &  0 \\
         1 & 10791 & 5781 & 2184 &   0 &   0 &   0 &  0 &  0 & 0 & 0 &  0 \\
         2 &  5010 & 3597 & 2184 & 771 &   0 &   0 &  0 &  0 & 0 & 0 &  0 \\
         3 &  1413 & 1413 & 1413 & 771 & 300 &   0 &  0 &  0 & 0 & 0 &  0 \\
         4 &     0 &    0 &  642 & 471 & 300 & 129 & 36 &  0 & 0 & 0 &  0 \\
         5 &     0 &    0 &  171 & 171 & 171 &  93 & 36 &  0 & 0 & 0 &  0 \\
         6 &     0 &    0 &    0 &   0 &  78 &  57 & 36 & 15 & 4 & 0 &  0 \\
         7 &     0 &    0 &    0 &   0 &  21 &  21 & 21 & 11 & 4 & 0 &  0 \\
         8 &     0 &    0 &    0 &   0 &   0 &   0 & 10 &  7 & 4 & 1 &  0 \\
         9 &     0 &    0 &    0 &   0 &   0 &   0 &  3 &  3 & 3 & 1 &  0 \\
        10 &     0 &    0 &    0 &   0 &   0 &   0 &  0 &  0 & 2 & 1 &  0 \\
        11 &     0 &    0 &    0 &   0 &   0 &   0 &  0 &  0 & 1 & 1 &  1 \\
    \end{tabular}
    \caption{Values of $c_{i, j}$ for $0\le i\le 11$ and $0\le j\le 10$.}
    \label{fig:tij-bot-a}
\end{figure}
\begin{figure}[h]
    \centering
    \begin{tabular}{c|*{11}{r}}
        \diaghead(1,-1){\hskip 0.03\hsize}{$i$}{$j$}
           &     0 &    1 &    2 &   3 &   4 &   5 &  6 &  7 & 8 & 9 & 10 \\ \hline
         0 & 19212 & 6702 &    0 &   0 &   0 &   0 &  0 &  0 & 0 & 0 &  0 \\
         1 & 12510 & 6702 & 2532 &   0 &   0 &   0 &  0 &  0 & 0 & 0 &  0 \\
         2 &  5808 & 4170 & 2532 & 894 &   0 &   0 &  0 &  0 & 0 & 0 &  0 \\
         3 &  1638 & 1638 & 1638 & 894 & 348 &   0 &  0 &  0 & 0 & 0 &  0 \\
         4 &     0 &    0 &  744 & 546 & 348 & 150 & 42 &  0 & 0 & 0 &  0 \\
         5 &     0 &    0 &  198 & 198 & 198 & 108 & 42 &  0 & 0 & 0 &  0 \\
         6 &     0 &    0 &    0 &   0 &  90 &  66 & 42 & 18 & 5 & 0 &  0 \\
         7 &     0 &    0 &    0 &   0 &  24 &  24 & 24 & 13 & 5 & 0 &  0 \\
         8 &     0 &    0 &    0 &   0 &   0 &   0 & 11 &  8 & 5 & 2 &  0 \\
         9 &     0 &    0 &    0 &   0 &   0 &   0 &  3 &  3 & 3 & 2 &  1 \\
        10 &     0 &    0 &    0 &   0 &   0 &   0 &  0 &  0 & 1 & 1 &  1 \\
        11 &     0 &    0 &    0 &   0 &   0 &   0 &  0 &  0 & 0 & 0 &  0 \\
    \end{tabular}
    \caption{Values of $d_{i, j}$ for $0\le i\le 11$ and $0\le j\le 10$.}
    \label{fig:tij-bot-b}
\end{figure}

We just need to compute the fraction of valid paths that go above and below each $C_{i, j}$, where $1\le i\le 2n + 21$ and $1\le j\le 2n + 20$. As noted earlier, these correspond to the probabilities $\mathbb{P}(x\prec y)$ associated to the poset $T_n$. If $C_{i, j}$ is a colored cell, then the desired fractions are trivially $0$ and $1$ in some order and we need not consider $C_{i, j}$. Therefore we can suppose that $C_{i, j}$ is uncolored.

Recall that the rectangle $R_m$ in the grid diagram of $T_n$ was composed of the $6$ cells $C_{i, j}$ for $2m + 9 \le i\le 2m + 11$ and $2m + 9\le j\le 2m + 10$.

Now, there are two cases to consider: $C_{i, j}$ is uncolored with $1\le i\le 11$ and $1\le j\le 10$, or $C_{i, j}$ is in some rectangle $R_m$ for $1\le m\le n$ where $(i, j)\in\{(2m + 9, 2m + 9), (2m + 9, 2m + 10), (2m + 10, 2m + 10)\}$. The first case, we can see, consists of $27$ different cells to explicitly consider.

We only need to consider half of the uncolored cells because of the rotational symmetry of the grid diagram of $T_n$.

\subsection{The First \texorpdfstring{$27$}{27} Cells}
This is the case $C_{i, j}$ for $1\le i\le 11$ and $1\le j\le 10$. We compute either the number of valid paths going above or below $C_{i, j}$ as indicated in the case.

As an example, we compute the number of valid paths going under $C_{6, 5}$. Notice that every valid path goes through exactly one of the points $(6, 4)$, $(5, 5)$, or $(4, 6)$. Furthermore, the valid paths going below $C_{6, 5}$ are precisely those going through $(6, 4)$. Thus our desired count is $p_{6, 4} = 69\cdot 78a_{n + 1} + 69\cdot 90b_{n + 1}$.

Figure~\ref{fig:pij-top} lists the number of valid paths going either below or above $C_{i, j}$, for $1\le i\le 11$ and $1\le j\le 10$ such that $C_{i, j}$ is uncolored.

\begin{figure}[h]
    \centering
    \begin{tabular}{c| c r}
        cell       & above or below & valid path count \\ \hline
        $C_{1, 1}$   & above & $5781a_{n + 1} + 6702b_{n + 1}$ \\
        $C_{2, 1}$   & below & $5010a_{n + 1} + 5808b_{n + 1}$ \\
        $C_{2, 2}$   & above & $4368a_{n + 1} + 5064b_{n + 1}$ \\
        $C_{3, 1}$   & below & $1413a_{n + 1} + 1638b_{n + 1}$ \\
        $C_{3, 2}$   & below & $5652a_{n + 1} + 6552b_{n + 1}$ \\
        $C_{3, 3}$   & above & $3855a_{n + 1} + 4470b_{n + 1}$ \\
        $C_{4, 3}$   & below & $5778a_{n + 1} + 6696b_{n + 1}$ \\
        $C_{4, 4}$   & above & $4200a_{n + 1} + 4872b_{n + 1}$ \\
        $C_{5, 3}$   & below & $1539a_{n + 1} + 1782b_{n + 1}$ \\
        $C_{5, 4}$   & below & $5472a_{n + 1} + 6336b_{n + 1}$ \\
        $C_{5, 5}$   & above & $4773a_{n + 1} + 5550b_{n + 1}$ \\
        $C_{5, 6}$   & above & $1332a_{n + 1} + 1554b_{n + 1}$ \\
        $C_{6, 5}$   & below & $5382a_{n + 1} + 6210b_{n + 1}$ \\
        $C_{6, 6}$   & above & $5148a_{n + 1} + 6006b_{n + 1}$ \\
        $C_{7, 5}$   & below & $1449a_{n + 1} + 1656b_{n + 1}$ \\
        $C_{7, 6}$   & below & $5124a_{n + 1} + 5856b_{n + 1}$ \\
        $C_{7, 7}$   & above & $4770a_{n + 1} + 5724b_{n + 1}$ \\
        $C_{7, 8}$   & above & $1272a_{n + 1} + 1590b_{n + 1}$ \\
        $C_{8, 7}$   & below & $5620a_{n + 1} + 6182b_{n + 1}$ \\
        $C_{8, 8}$   & above & $4792a_{n + 1} + 5990b_{n + 1}$ \\
        $C_{9, 7}$   & below & $1686a_{n + 1} + 1686b_{n + 1}$ \\ 
        $C_{9, 8}$   & below & $6012a_{n + 1} + 6012b_{n + 1}$ \\
        $C_{9, 9}$   & above & $2640a_{n + 1} + 5280b_{n + 1}$ \\
        $C_{10, 9}$  & below & $9288a_{n + 1} + 4644b_{n + 1}$ \\
        $C_{10, 10}$ & above & $0000a_{n + 1} + 7284b_{n + 1}$ \\
        $C_{11, 9}$  & below & $4644a_{n + 1} + 0000b_{n + 1}$ \\
        $C_{11, 10}$ & below & $16572a_{n +1} + 0000b_{n + 1}$ \\
    \end{tabular}
    \caption{Number of valid paths going either above or below $C_{i, j}$.}
    \label{fig:pij-top}
\end{figure}

To finish the case of $1\le i\le 11$ and $1\le j\le 10$, we just need to show that each of these counts is less than the count in the first row. We already know that the first row is $p_{0, 1} < \frac{p}{2}$, and since $\frac{p_{0, 1}}{p}$ is our claimed balance constant, proving this will finish.

By inspection, all the rows except the row corresponding to $C_{9, 8}$, $C_{10, 9}$, $C_{10, 10}$, and $C_{11, 10}$ give counts clearly less than the top row. We need to show that $6012a_{n + 1} + 6012b_{n + 1}$, and $9288a_{n + 1} + 4644b_{n + 1}$, and $7284b_{n + 1}$, and $16572a_{n + 1}$ are each less than $5781a_{n + 1} + 6702b_{n + 1}$. If we write $f = \frac{a_{n + 1}}{b_{n + 1}}$, then these four inequalities are respectively equivalent to $f < \frac{230}{77}$, and $f < \frac{98}{167}$, and $f > \frac{194}{1927}$, and $f < \frac{2234}{3597}$.

Thus we need to show $\frac{a_{n + 1}}{b_{n + 1}}\in\left(\frac{194}{1927}, \frac{98}{167}\right)$. Recall that $(a_1, b_1) = (19212, 35784)$, and $a_{m + 1} = 3a_m + 3b_m$, and $b_{m + 1} = 4a_m + 6b_m$. Thus, letting $f_i = \frac{a_i}{b_i}$, we find $f_1 = \frac{19212}{35784}$ and $f_{m + 1} = \frac{3f_m + 3}{4f_m + 6}$. It is not hard to show that $f_i$ is a strictly increasing positive sequence with limit $\frac{-3 + \sqrt{57}}{8}$. Furthermore, clearly $f_1$ and the limit are within the required interval, which establishes the desired result.

\subsection{The Cell \texorpdfstring{$C_{i, j}$}{Cij} is in \texorpdfstring{$R_m$}{Rm}}
This is the final case to check. As remarked earlier, we only need to consider $(i, j)\in\{(2m + 9, 2m + 9), (2m + 9, 2m + 10), (2m + 10, 2m + 10)\}$, which is half of the cells in all the $R_m$ rectangles, because we can capitalize on the rotational symmetry of the grid diagram of $T_n$.

\begin{case2}\label{case2:1}
$(i, j) = (2m + 9, 2m + 9)$
\end{case2}
Then $C_{i, j}$ is the bottom left cell of $R_m$. We will compute the number of paths going above $C_{i, j}$, which we see is precisely $p_{i - 1, j} = p_{2m + 8, 2m + 9}$, that is, the number of paths through the top left corner of $C_{i, j}$.

We see, referencing Figure~\ref{fig:constructed-diagram}, that $t_{2m + 8, 2m + 9} = a_m$. Also, $r_{2m + 8, 2m + 9} = t_{2n - 2m + 13, 2n - 2m + 11} = a_{n + 1 - m} + 3b_{n + 1 - m}$. (This value is depicted as $e_{n + 1 - m}$ in Figure~\ref{fig:constructed-diagram}.)

Recall that it suffices to show $p_{i - 1, j} < p_{0, 1}$, or
\[a_m(a_{n + 1 - m} + 3b_{n + 1 - m}) < 5781a_{n + 1} + 6702b_{n + 1},\]
subject to the condition $1\le m\le n$.

It is sufficient to show that $a_m(a_{n + 1 - m} + 3b_{n + 1 - m})\le (2a_m + b_m)b_{n - m + 1}$, which then reduces the desired inequality to one that we later prove in Case~\ref{case2:3}.

This new inequality is equivalent to
\begin{align*}
a_m(a_{n + 1 - m} + b_{n + 1 - m})&\le b_mb_{n - m + 1},\\
\frac{a_m}{b_m}\cdot\left(\frac{a_{n - m + 1}}{b_{n - m + 1}} + 1\right)&\le 1.
\end{align*}
Recall from earlier that $f_i = \frac{a_i}{b_i}$ for $i\ge 1$ is an increasing positive sequence with limit $\frac{-3 + \sqrt{57}}{8}$. Since
\[\left(\frac{-3 + \sqrt{57}}{8}\right)\cdot\left(\frac{-3 + \sqrt{57}}{8} + 1\right) < 1,\]
the desired inequality is true.

\begin{case2}\label{case2:2}
$(i, j) = (2m + 9, 2m + 10)$
\end{case2}
Then $C_{i, j}$ is the top left cell of $R_m$. We will compute the number of paths going above $C_{i, j}$, which we see is precisely $p_{i - 1, j} = p_{2m + 8, 2m + 10}$.

We see, referencing Figure~\ref{fig:constructed-diagram}, that $t_{2m + 8, 2m + 10} = a_m$. Also, $r_{2m + 8, 2m + 10} = t_{2n - 2m + 13, 2n - 2m + 10} = b_{n - m + 1}$.

It suffices to show $p_{i - 1, j} < p_{0, 1}$, or
\[a_mb_{n - m + 1} < 5781a_{n + 1} + 6702b_{n + 1},\]
subject to the condition $1\le m\le n$. This inequality is clearly weaker than that of Case~\ref{case2:3}, so we again defer to that case.

\begin{case2}\label{case2:3}
$(i, j) = (2m + 10, 2m + 10)$
\end{case2}
Then $C_{i, j}$ is the middle top cell of $R_m$. We will compute the number of paths going above $C_{i, j}$, which we see is precisely $p_{i - 1, j} = p_{2m + 9, 2m + 10}$.

We see, referencing Figure~\ref{fig:constructed-diagram}, that $t_{2m + 9, 2m + 10} = a_m + 2b_m$. Additionally, $r_{2m + 9, 2m + 10} = t_{2n - 2m + 12, 2n - 2m + 10} = b_{n - m + 1}$.

It suffices to show $p_{i - 1, j} < p_{0, 1}$, or
\[(2a_m + b_m)b_{n - m + 1} < 5781a_{n + 1} + 6702b_{n + 1},\]
subject to the condition $1\le m\le n$.

Recall that $f_i = \frac{a_i}{b_i}$ for $i\ge 1$ is an increasing positive sequence with limit $\frac{-3 + \sqrt{57}}{8}$. Hence
\begin{align*}
(2a_m + b_m)b_{n - m + 1} &< \left(2\cdot\frac{-3 + \sqrt{57}}{8} + 1\right)b_mb_{n - m + 1}\\
&\le\frac{1 + \sqrt{57}}{4}b_1b_n\\
&\le 44151a_n + 57555b_n\\
&= 5781a_{n + 1} + 6702b_{n + 1}.
\end{align*}
The first inequality follows from $f_m <\frac{-3 + \sqrt{57}}{8}$, the second inequality follows from the log-convexity of $(b_i)_{i = 1}^n$, the third inequality follows from $b_1 = 35784$ and
\[\frac{-16203 + 2982\sqrt{57}}{14717} < \frac{19212}{35784}\le f_n,\]
and the last relation is an equality, using the recurrence $a_{n + 1} = 3a_n + 3b_n$ and $b_{n + 1} = 4a_n + 6b_n$.

The reason $(b_i)_{i = 1}^n$ is log-convex is that for positive integers $i$ we have
\[b_i = \frac{(12906 + 2542\sqrt{57})\left(\frac{9 + \sqrt{57}}{2}\right)^i + (-12906 + 2542\sqrt{57})\left(\frac{9 - \sqrt{57}}{2}\right)^i}{\sqrt{57}},\]
where $-12906 + 2542\sqrt{57} > 0$. We can then actually check that the function $b:\mathbb{R}^+\rightarrow\mathbb{R}^+$ defined by
\[b(x) = \log\frac{(12906 + 2542\sqrt{57})\left(\frac{9 + \sqrt{57}}{2}\right)^x + (-12906 + 2542\sqrt{57})\left(\frac{9 - \sqrt{57}}{2}\right)^x}{\sqrt{57}}\]
when $x > 0$ satisfies $b''(x) > 0$ for all $x > 0$, which implies the desired log-convexity at values $b_i = b(i)$, where $1\le i\le n$. (In general, a function $\log (\alpha_1\cdot\alpha_2^x + \alpha_3\cdot\alpha_4^x)$ has second derivative $\frac{\alpha_1\alpha_3(\alpha_2\alpha_4)^x(\log\alpha_2 - \log\alpha_4)^2}{(\alpha_1\alpha_2^x + \alpha_3\alpha_4^x)^2} > 0$.)

\subsection{Conclusion}
All the cases are complete, finishing our computations. Thus, indeed,
\[\delta(T_n) = \frac{5781a_{n + 1} + 6702b_{n + 1}}{16572a_{n + 1} + 19212b_{n + 1}}\rightarrow\frac{5864893 + 27\sqrt{57}}{16812976},\]
as claimed.

\bibliographystyle{plainnat}
\bibliography{main}

\begin{thebibliography}{20}
\providecommand{\natexlab}[1]{#1}
\providecommand{\url}[1]{\texttt{#1}}
\expandafter\ifx\csname urlstyle\endcsname\relax
  \providecommand{\doi}[1]{doi: #1}\else
  \providecommand{\doi}{doi: \begingroup \urlstyle{rm}\Url}\fi

\bibitem[Aigner(1985)]{aigner1985note}
Martin Aigner.
\newblock A note on merging.
\newblock \emph{Order}, 2\penalty0 (3):\penalty0 257--264, 1985.

\bibitem[Brightwell et~al.(1995)Brightwell, Felsner, and
  Trotter]{brightwell1995balancing}
G.~R. Brightwell, S.~Felsner, and W.~T. Trotter.
\newblock Balancing pairs and the cross product conjecture.
\newblock \emph{Order}, 12\penalty0 (4):\penalty0 327--349, 1995.

\bibitem[Brightwell(1999)]{brightwell1999balanced}
Graham Brightwell.
\newblock Balanced pairs in partial orders.
\newblock \emph{Discrete Mathematics}, 201\penalty0 (1-3):\penalty0 25--52,
  1999.

\bibitem[Brightwell(1989)]{brightwell1989semiorders}
Graham~R. Brightwell.
\newblock Semiorders and the 1/3--2/3 conjecture.
\newblock \emph{Order}, 5\penalty0 (4):\penalty0 369--380, 1989.

\bibitem[Cardinal et~al.(2013)Cardinal, Fiorini, Joret, Jungers, and
  Munro]{cardinal2013sorting}
Jean Cardinal, Samuel Fiorini, Gwena\"{e}l Joret, Rapha\"{e}l~M. Jungers, and
  J.~Ian Munro.
\newblock Sorting under partial information (without the ellipsoid algorithm).
\newblock \emph{Combinatorica}, 33\penalty0 (6):\penalty0 655--697, 2013.
\newblock ISSN 0209-9683.
\newblock \doi{10.1007/s00493-013-2821-5}.

\bibitem[Chen(2018)]{chen2018family}
Evan Chen.
\newblock A family of partially ordered sets with small balance constant.
\newblock \emph{Electron. J. Combin.}, 25\penalty0 (4):\penalty0 Paper 4.43,
  13, 2018.

\bibitem[Fiorini and Rexhep(2016)]{fiorini2016entropy}
Samuel Fiorini and Selim Rexhep.
\newblock Poset entropy versus number of linear extensions: the width-2 case.
\newblock \emph{Order}, 33\penalty0 (1):\penalty0 1--21, 2016.
\newblock ISSN 0167-8094.
\newblock \doi{10.1007/s11083-015-9346-z}.

\bibitem[Fishburn(1976)]{fishburn1976linear}
Peter~C Fishburn.
\newblock On linear extension majority graphs of partial orders.
\newblock \emph{Journal of Combinatorial Theory, Series B}, 21\penalty0
  (1):\penalty0 65--70, 1976.

\bibitem[Fredman(1976)]{fredman1976good}
Michael~L Fredman.
\newblock How good is the information theory bound in sorting?
\newblock \emph{Theoretical Computer Science}, 1\penalty0 (4):\penalty0
  355--361, 1976.

\bibitem[Hoggar(1974)]{hoggar1974chromatic}
S.~G. Hoggar.
\newblock Chromatic polynomials and logarithmic concavity.
\newblock \emph{Journal of Combinatorial Theory, Series B}, 16\penalty0
  (3):\penalty0 248--254, 1974.

\bibitem[Kahn and Kim(1995)]{kahn1995entropy}
Jeff Kahn and Jeong~Han Kim.
\newblock Entropy and sorting.
\newblock volume~51, pages 390--399. 1995.
\newblock \doi{10.1006/jcss.1995.1077}.
\newblock 24th Annual ACM Symposium on the Theory of Computing (Victoria, BC,
  1992).

\bibitem[Kahn and Saks(1984)]{kahn1984balancing}
Jeff Kahn and Michael Saks.
\newblock Balancing poset extensions.
\newblock \emph{Order}, 1\penalty0 (2):\penalty0 113--126, 1984.

\bibitem[Kislitsyn(1968)]{kislitsyn1968finite}
S.~S. Kislitsyn.
\newblock Finite partially ordered sets and their corresponding permutation
  sets.
\newblock \emph{Math. Notes}, 4:\penalty0 798--801, 1968.

\bibitem[Linial(1984)]{linial1984information}
Nathan Linial.
\newblock The information-theoretic bound is good for merging.
\newblock \emph{SIAM Journal on Computing}, 13\penalty0 (4):\penalty0 795--801,
  1984.

\bibitem[Olson and Sagan(2018)]{olson2018on}
Emily~J. Olson and Bruce~E. Sagan.
\newblock On the 1/3--2/3 conjecture.
\newblock \emph{Order}, 35\penalty0 (3):\penalty0 581--596, 2018.
\newblock ISSN 0167-8094.
\newblock \doi{10.1007/s11083-017-9450-3}.

\bibitem[Peczarski(2008)]{peczarski2008gold}
Marcin Peczarski.
\newblock The {G}old {P}artition {C}onjecture for 6-thin {P}osets.
\newblock \emph{Order}, 25\penalty0 (2):\penalty0 91--103, 2008.

\bibitem[Peczarski(2017)]{peczarski2017worst}
Marcin Peczarski.
\newblock The {W}orst {B}alanced {P}artially {O}rdered {S}ets-{L}adders with
  {B}roken {R}ungs.
\newblock \emph{Experimental Mathematics}, pages 1--4, 2017.

\bibitem[Trotter et~al.(1992)Trotter, Gehrlein, and
  Fishburn]{trotter1992balance}
W.~T. Trotter, W.~G. Gehrlein, and P.~C. Fishburn.
\newblock Balance theorems for height-2 posets.
\newblock \emph{Order}, 9\penalty0 (1):\penalty0 43--53, 1992.

\bibitem[Zaguia(2012)]{zaguia20121}
Imed Zaguia.
\newblock The {$1/3$}-{$2/3$} conjecture for {$N$}-free ordered sets.
\newblock \emph{Electron. J. Combin.}, 19\penalty0 (2):\penalty0 Paper 29, 5,
  2012.

\bibitem[Zaguia(2019)]{zaguia20191}
Imed Zaguia.
\newblock The 1/3-2/3 conjecture for ordered sets whose cover graph is a
  forest.
\newblock \emph{Order}, 36\penalty0 (2):\penalty0 335--347, 2019.
\newblock ISSN 0167-8094.
\newblock \doi{10.1007/s11083-018-9469-0}.

\end{thebibliography}

\end{document}